\documentclass[11pt]{article}

\usepackage[paper=a4paper,left=30mm,right=30mm,top=30mm,bottom=30mm]{geometry}
\usepackage{amsmath} 
\usepackage{amssymb} 
\usepackage{amsthm} 
\usepackage{graphicx}
\usepackage{algorithm} 
\usepackage{algorithmic} 

\usepackage[shortlabels]{enumitem}
    \setlist{nosep}
\usepackage{todonotes}
\usepackage{nicematrix}
\usepackage{mathtools}
\usepackage{microtype}

\usepackage{xspace}
\usepackage[all]{nowidow}
\usepackage{cite}
\usepackage{hyperref}
    \hypersetup{hidelinks=true}
\usepackage[nameinlink]{cleveref} 

\newcommand{\norm}[1]{\Vert #1 \Vert}

\newcommand{\RR}{\mathbb{R}}
\newcommand{\MM}{\mathcal{M}}
\newcommand{\ham}{\mathcal{H}}
\newcommand{\etabar}{\overline{\eta}}

\newcommand{\und}{,~}

\DeclareMathOperator{\Span}{span}

\newcommand{\inv}{^{-1}}
\newcommand{\tp}{^{\mathsf{T}}}

\AtBeginDocument{\renewcommand{\tilde}{\widetilde}}
\AtBeginDocument{\renewcommand{\hat}{\widehat}}

\DeclareMathOperator*{\ran}{ran}
\DeclareMathOperator*{\rank}{rank}

\newcommand{\dt}{\,\mathrm{d}t}
\newcommand{\ddt}{\tfrac{\mathrm{d}}{\mathrm{d}t}}
\newcommand{\dx}{\,\mathrm{d}x}

\newcommand{\ds}{\,\mathrm{d}s}

\newcommand{\pH}{pH\xspace}
\newcommand{\portHamiltonian}{port-Ha\-mil\-to\-ni\-an\xspace}

\renewcommand{\vec}[1]{\begin{bmatrix} #1 \end{bmatrix}}

\newcommand{\tvec}[1]{\left[\begin{smallmatrix} #1 \end{smallmatrix}\right]}

\newcommand{\Bigtextvec}[1]{\Big[\begin{smallmatrix} #1 \end{smallmatrix}\Big]}

\renewcommand{\implies}{\!\!$\Rightarrow$}

\newcommand{\sym}{_{\mathsf{H}}}
\renewcommand{\skew}{_{\mathsf{S}}}

\newcommand{\bfpj}{\mathbf{p}}
\newcommand{\bfTj}{\mathbf{T}}
\newcommand{\pj}{p}
\newcommand{\Tj}{T}

\newcommand{\subs}{\mathcal{W}}
\newcommand{\proj}{\mathcal{P}}
\newcommand{\cproj}{\mathcal{P}_{\!c}}

\newtheorem{theorem}{Theorem}

\newtheorem{definition}[theorem]{Definition}
\newtheorem{example}[theorem]{Example}
\newtheorem{lemma}[theorem]{Lemma}
\newtheorem{proposition}[theorem]{Proposition}
\newtheorem{remark}[theorem]{Remark}
\newtheorem{assumption}[theorem]{Assumption}

\begin{document}

\title{Passivity encoding representations of nonlinear systems}
\author{A. Karsai, T. Breiten, J. Ramme, P. Schulze
\thanks{
    This work was supported in part by the Deutsche Forschungsgemeinschaft within the subproject B03 in the Sonderforschungsbereich/Transregio 154 “Mathematical Modelling, Simulation and Optimization using the Example of Gas Networks” (Project 239904186).
}
\thanks{
    A. Karsai (email: karsai@math.tu-berlin.de), T. Breiten (email: tobias.breiten@tu-berlin.de), J. Ramme (email: ramme@math.tu-berlin.de) and P. Schulze (email: pschulze@math.tu-berlin.de) are all with the Institute of Mathematics, Technische Universität Berlin, 10623 Berlin, Germany.
}
}

\maketitle

\begin{abstract}
    Passive systems are characterized by their inability to generate energy internally, providing a powerful tool for modeling physical phenomena.
    Additionally, algebraically encoding passivity in the system description can be advantageous.
    For this, \portHamiltonian systems are a prominent approach.
    Another possibility is writing the system in suitable coordinates.
    In this paper, we investigate the equivalence between passivity and the feasibility of passivity encoding representations, thereby elaborating upon existing results for \portHamiltonian systems.
    Based on our findings, we present a method to construct \portHamiltonian representations of a passive system if the dynamics and the Hamiltonian are known.
\end{abstract}

\section{Introduction}

Physical processes can often be modeled by means of differential equations.
Furthermore, many real-world phenomena have additional energy properties associated with the dynamics.
Such a property is \emph{dissipativity}, which was extensively studied by J.\,Willems in the seminal works~\cite{willems72-dissipative1,willems72-dissipative2}.
Dissipative systems are characterized by their inability to generate energy internally.
This behavior is modeled by a \emph{storage function} which can only increase according to a given \emph{supply rate}.
A supply rate with particular importance is the \emph{impedance supply}.
Systems that are dissipative with respect to this supply rate are called \emph{passive}, and the storage function is called \emph{Hamiltonian}.
Although Willems' definition extends to systems without state space representations, we focus on finite dimensional state space systems.

For linear systems, passivity can be characterized using the Kalman--Yakubovich--Popov (KYP) inequality, see, e.g.,~\cite{willems72-dissipative2}.
Under smoothness and controllability assumptions, the result was extended to nonlinear systems in~\cite{moylan74-implications}, see also \cite{brockett66-path,willems72-dissipative1}.
The results of~\cite{moylan74-implications} were subsequently extended to dissipative systems with quadratic supply rates in~\cite{hill80-dissipative}. 

For modeling purposes, encoding additional algebraic properties in the system dynamics can be advantageous.
A prominent approach are \portHamiltonian (\pH) systems~\cite{mehrmann23-control,vanderschaft14-port,vanderschaft17-l2gain}, which are motivated by the geometric viewpoint of Dirac structures~\cite{vanderschaft14-port} and explicitly include the gradient of the Hamiltonian in the dynamics.
Other model classes algebraically encoding passivity include the framework of \emph{monotone} \portHamiltonian systems~\cite{camlibel23-port,gernandt24-port} and the related approach~\cite{giesselmann24-energy}.

First steps towards understanding the relation between passivity and the \portHamiltonian structure date back to~\cite{willems72-dissipative2}, see~\cite{cherifi24-difference} for a recent overview on the linear case.
For the recently introduced structure from~\cite{giesselmann24-energy}, this relationship is still largely unexplored.
Another question is the construction of algebraic representations in either of the model classes for systems known to be passive with a given Hamiltonian.
In the linear case, it is well known how to construct \portHamiltonian representations, see, e.g.,~\cite{willems72-dissipative2,prajna02-lmi}.
For the nonlinear case, first ideas appeared in~\cite{mclachlan99-geometric}, where the possibility to express a passive system via a linear, state dependent action on the gradient of the Hamiltonian was investigated without focus on \portHamiltonian structure.
Later, the \pH structure was incorporated in~\cite{ortega02-interconnection, wang03-generalized}, where the authors used methods similar to~\cite{mclachlan99-geometric}. 
Unfortunately, the constructions from~\cite{ortega02-interconnection, wang03-generalized} exhibit notable drawbacks. 
Firstly, they fail to produce linear representations if the original dynamics are linear. 
Secondly, and more critically, the constructed operators may exhibit singularities whenever the gradient of the Hamiltonian vanishes.
While~\cite{mclachlan99-geometric,ortega02-interconnection, wang03-generalized} all propose potential remedies, this remains a difficult problem.

Our main contributions are as follows:
\begin{itemize}
\item
    In \Cref{thm:P-B}, we characterize the class considered by~\cite{giesselmann24-energy} and show that a sufficient condition for passive systems to admit a representation in the class is the injectivity of the gradient of the Hamiltonian.
    Further, in \Cref{thm:B-A,thm:A-B} we analyze the relationship of this class and \portHamiltonian systems.
\item
    In \Cref{thm:P-A,thm:P-A-fix}, we adapt the construction of~\cite{mclachlan99-geometric} to \portHamiltonian systems, in turn making the results more easily accessible, and propose an approach to overcome a difficulty in the construction.
    See also \Cref{rem:P-A-mclachlan} on the differences to~\cite{mclachlan99-geometric}.
\item
    In \Cref{sec:constructions}, we focus on \portHamiltonian systems and propose a framework to overcome difficulties of the previously discussed constructions.
    We state \Cref{thm:conservative-systems,thm:P-general}, which deal with the conservative and general cases respectively.
\end{itemize}

The paper is organized as follows.
In \Cref{sec:problem-setting} we state the considered model classes and clarify the problem setting.
In \Cref{sec:characterizations} we investigate conditions under which a passive system can be formulated in either of the structures and comment on their relationship. 
\Cref{sec:constructions} contains the previously mentioned framework to overcome difficulties of the constructions for \portHamiltonian systems.
In \Cref{sec:examples}, we illustrate our findings with several examples, including applications to both finite and infinite dimensional systems.
Finally, in \Cref{sec:conclusion} we present our conclusions and provide an outlook for future research.

\subsection*{Notation}
We denote the set of all $k$-times continuously differentiable functions from $U$ to $V$ by $C^k(U,V)$ and define $C(U,V) \coloneq C^0(U,V)$.
When the spaces $U$ and $V$ are clear from context, we abbreviate $C^k \coloneq C^k(U,V)$.
The Jacobian of a function~$f\colon \RR^n \to \RR^n$ at a point $z$ is denoted by $Df(z)$.
For a matrix $A \in \RR^{n,n}$, we denote the skew-symmetric part and symmetric part by $A\skew\coloneq \frac12 (A-A\tp)$ and $A\sym\coloneq \frac12 (A + A\tp)$, respectively, where $\cdot\tp$ denotes the transpose. 
The kernel and range of the matrix~$A$ are denoted by $\ker(A)$ and $\ran(A)$, respectively.
Further, we write $A \succeq 0$ (resp.~$A\succ 0$) if $z\tp A z \geq 0$ (resp.~$z\tp A z > 0$) for all $z \in \RR^n\setminus\{0\}$.

\section{Problem setting}
\label{sec:problem-setting}
We consider systems of the form
\begin{equation}\label{eq:model}
    \begin{aligned}
        \ddt{z}(t) & = f(z(t)) + g(z(t)) u(t) \\
        y(t) & = h(z(t)),
    \end{aligned}
\end{equation}
where we do not incorporate feedthrough terms for simplicity.
Here, at time $t \in \RR$, the vector $z(t)\in\RR^n$ is the state of the system, $u(t)\in\RR^m$ is a control input, and $y(t)\in\RR^m$ is a passive output not necessarily corresponding to real world measurements.
In what follows, we suppress the time dependency and write $z$, $u$ and $y$ instead of $z(t)$, $u(t)$ and $y(t)$.
Further, we abbreviate $\dot{z} = \ddt z$.
For smooth storage functions, the term passivity~\cite{willems72-dissipative1} refers to the fact that there exists a
Hamiltonian $\ham\colon \RR^n \to \RR$ which satisfies $\ham \geq 0$ and 
\begin{equation}\label{eq:energybalance}
    \ham(z(t_1)) - \ham(z(t_0)) \leq \int_{t_0}^{t_1} y\tp u \dt
\end{equation}
for all $t_1 \geq t_0$ and all controls $u$ along solutions of~\eqref{eq:model}, where $y\tp$ denotes the transpose of $y$.
In other words, $\ham$ and $y\tp u$ are the storage function and the supply rate, respectively.
In the following, we restrict our attention to the case that $\ham$ is continuously differentiable, so that its gradient $\eta \coloneq \nabla \ham$ is well-defined and continuous.
It is shown in \cite{moylan74-implications} that under smoothness and controllability assumptions system~\eqref{eq:model} is passive with Hamiltonian $\ham \geq 0$ if and only if there exists $\ell\colon \RR^n \to \RR^p$ with 
\begin{equation}\label{eq:lure}
    \begin{aligned}
        \eta(z)\tp f(z) & = - \ell(z)\tp \ell(z) \\
        g(z)\tp \eta(z) & = h(z),
    \end{aligned}
\end{equation}
see also \cite{brockett66-path,willems72-dissipative1}.

In this paper we consider \pH systems of the form
\begin{equation}\label{eq:ph-a}
    \begin{aligned}
        \dot{z} & = (J(z) - R(z))\eta(z) + B(z) u \\
        y & = B(z)\tp \eta(z),
    \end{aligned}
    \tag{\text{a}}
\end{equation}
where $J, R \colon \RR^n \to \RR^{n,n}$ with $J(z) = -J(z)\tp$, $R(z) = R(z)\tp \succeq 0$ for all $z\in \RR^n$ and $B \colon \RR^n \to \RR^{n,m}$.
The class~\eqref{eq:ph-a} corresponds to the special case of linear resistive structures, and the generalization to nonlinear resistive structures replaces $R$ by $\mathcal{R}\colon \RR^n \times \RR^n \to \RR^n$ with $v\tp \mathcal{R}(z,v) \geq 0$ for all $z, v \in \RR^n$, see~\cite[Definition 6.1.4]{vanderschaft17-l2gain}.
As mentioned before, other approaches to model nonlinear actions on $\eta$ include monotone \portHamiltonian systems~\cite{camlibel23-port,gernandt24-port} and the framework of~\cite{giesselmann24-energy}.
The latter considers models of the form
\begin{equation*}
    \dot{z} = j(\eta(z)) - r(\eta(z)) + b(\cdot, \eta(z)),
\end{equation*}
where $j, r \colon \eta(\RR^n) \to \RR^n$ with $v\tp j(v) = 0$ and $v\tp r(v) \geq 0$ for all $v \in \eta(\RR^n)$, and the term $b(t,\eta(z(t)))$ models control inputs at time $t$.
In the following, we focus on the case in which the explicit time dependency in $b(\cdot,\eta(z))$ is solely attributable to an external control variable~$u$, and that $b(\cdot, \eta(z))$ is linear in~$u$.
In this case, we can write $b(\cdot,\eta(z)) = B(z) u$ for some $B\colon \RR^n \to \RR^{n,m}$, and we arrive at the model
\begin{equation}\label{eq:ph-b}
    \begin{aligned}
        \dot{z} & = j(\eta(z)) - r(\eta(z)) + B(z) u \\
        y & = B(z)\tp \eta(z).
    \end{aligned}
    \tag{\text{b}}
\end{equation}
If $r$ is the gradient of a convex function and $j(\eta(z))$ is replaced by $J \eta(z)$ with $J\in\RR^{n,n}$ being a skew-symmetric matrix, then we recover the class from~\cite[Equation 4.10]{camlibel23-port}.
We refer to \cite{camlibel23-port} for examples of systems naturally modeled in the form~\eqref{eq:ph-b} and to \Cref{rem:ph-b-and-nonlinear-resistive-structures} for comments on advantages of the model class~\eqref{eq:ph-b} over nonlinear resistive structures $\mathcal{R}$ as above.
In both of the models~\eqref{eq:ph-a} and~\eqref{eq:ph-b}, the structural assumptions on the operators $J$ and $R$ (resp.\,$j$ and $r$) and $\eta = \nabla \ham$ ensure that the system is passive with Hamiltonian $\ham$, see \Cref{prop:ph-is-passive}.
Additionally, the algebraic properties ensure that coupling of these systems is easily possible in a structure-preserving manner.

Before discussing the relationship between passivity and the classes~\eqref{eq:ph-a} and~\eqref{eq:ph-b}, some remarks on the uniqueness of the representations and the relevance of the class~\eqref{eq:ph-b} are in order.

\begin{remark}[uniqueness]\label{rem:uniqueness}
    In general, we do not have uniqueness in any of the previously mentioned representations.
    For the class~\eqref{eq:ph-a}, this fact is well known~\cite{mclachlan99-geometric} and will be exploited in~\Cref{sec:constructions}.
    For systems of the form~\eqref{eq:ph-b}, we can always choose $\tilde{j} \coloneq 0$ and $\tilde{r} \coloneq j - r$ such that $v\tp \tilde{j}(v) = 0$ and $v\tp \tilde{r}(v) \geq 0$ for all $v \in \eta(\RR^n)$.
    Similarly, for the \pH representation with nonlinear resistive structures, i.e.~$f(z) = J(z) \eta(z) - \mathcal{R}(z,\eta(z))$ with $v\tp \mathcal{R}(z,v) \geq 0$, we can always absorb conservative effects into $\mathcal{R}$.
    Nonetheless, it can be advantageous to include the functions~$j$ or $J$ in the formulation in order to emphasize the energy conservative parts of the dynamics.
\end{remark}

\begin{remark}\label{rem:ph-b-and-nonlinear-resistive-structures}
    Apart from being able to model nonlinear conservative actions on $\eta$, one benefit of modeling with~\eqref{eq:ph-b} rather than with nonlinear resistive structures as in~\cite[Definition 6.1.4]{vanderschaft17-l2gain} lies in the lack of redundancy: if $\eta$ is injective, then $\mathcal{R}(z,\eta(z))$ is uniquely determined by either of its inputs.
    Although this is no problem in symbolic calculations, it can be disadvantageous for numerical methods.
    To illustrate this, first notice that every passive system can be written in the form $f(z) = J(z) \eta(z) - \mathcal{R}(z,\eta(z))$ with $v\tp \mathcal{R}(z,v) \geq 0$ for all $z, v \in \RR^n$ by setting $J(z) \coloneq 0$ and
    \begin{equation*}
        \mathcal{R}(z,v)
        \coloneq 
        \begin{cases}
            - f(z) & \text{if $v = \eta(z)$,} \\
            0 & \text{else,}    
        \end{cases}
    \end{equation*}
    and that the above $\mathcal{R}$ can be approximated to arbitrary precision using $C^\infty$ functions without altering the corresponding dynamics.
    For a structured time discretization of the dynamics, discrete gradients~\cite{gonzalez96-time,mclachlan99-geometric} can be used.
    These are approximations $\etabar\colon \RR^n \times \RR^n \to \RR$ of $\eta = \nabla \ham$ satisfying 
    \begin{equation*}
        \etabar(z_1, z_2)\tp (z_2-z_1) = \ham(z_2) - \ham(z_1)
    \end{equation*}
    for all $z_1, z_2 \in \RR^n$.
    Similar to~\cite{breiten25-passive}, where the same idea was used for the structure~\eqref{eq:ph-b}, it is straightforward to show that the iterates $z_i$ defined by $z_0 = z(t_0)$ and
    \begin{equation*}
        \tfrac{z_{i+1} - z_{i}}{t_{i+1} - t_{i}} = -\mathcal{R}(\tfrac{z_{i} + z_{i+1}}{2}, \etabar(z_{i}, z_{i+1})) + B(\tfrac{z_{i} + z_{i+1}}{2}) u(\tfrac{t_{i} + t_{i+1}}{2})
    \end{equation*}
    satisfy an energy balance equation.
    However, unlike the scheme corresponding to~\eqref{eq:ph-b}, we cannot expect convergence of the iterates to the true solution in general because $\etabar(z_{i}, z_{i+1})$ is only an approximation to $\eta(\tfrac{z_{i} + z_{i+1}}{2})$ and, as outlined above,~$\mathcal{R}$ could be pathological.
\end{remark}

\section{Passivity encoding representations}
\label{sec:characterizations}

In order to facilitate the discussion, we assume that the second equation in~\eqref{eq:lure} is satisfied and focus on the system
\begin{equation}\label{eq:nonlinear-system}
    \begin{aligned}
    \dot{z} & = f(z) + B(z) u \\
    y & = B(z)\tp \eta(z).
    \end{aligned}
\end{equation}
Note that the Hamiltonian~$\ham$ and~$\eta = \nabla \ham$ are fixed in our discussion. 

\begin{definition}
    We define the three properties (P), (A) and (B) as follows.
    \begin{itemize}
    \item[(P)]
        The system~\eqref{eq:nonlinear-system} is passive with Hamiltonian $\ham$ such that equations~\eqref{eq:lure} hold.
    \item[(A)]
        The system~\eqref{eq:nonlinear-system} can be represented in the form~\eqref{eq:ph-a}.
        In other words, for all $z\in \RR^n$ we have $f(z) = (J(z) - R(z)) \eta(z)$, where $J(z) = -J(z)\tp$ and $R(z) = R(z)\tp \succeq 0$ for all $z \in \RR^n$.
    \item[(B)]
        The system~\eqref{eq:nonlinear-system} can be represented in the form~\eqref{eq:ph-b}.
        In other words, for all $z\in \RR^n$ we have $f(z) = j(\eta(z)) - r(\eta(z))$, where $v\tp j(v) = 0$ and $v\tp r(v) \geq 0$ for all $v \in \eta(\RR^n)$.
    \end{itemize}
\end{definition}

As we have mentioned in the introduction, systems of the form~\eqref{eq:ph-a} and~\eqref{eq:ph-b} are always passive.

\begin{proposition}[\!\!\cite{vanderschaft14-port,giesselmann24-energy}]\label{prop:ph-is-passive}
    Systems of the form~\eqref{eq:ph-a} or~\eqref{eq:ph-b} are passive, and trajectories~$z$ of~\eqref{eq:ph-a} or~\eqref{eq:ph-b} satisfy the {energy balance} 
    \begin{equation*}
        \ham(z(t_1)) - \ham(z(t_0)) = \int_{t_0}^{t_1} (-d(z) + y\tp u) \dt
    \end{equation*}
    for all $t_1 \geq t_0$, where $d(z) = \eta(z)\tp R(z) \eta(z) \geq 0$ or $d(z) = \eta(z)\tp r(\eta(z)) \geq 0$, respectively.
    In particular, any of the conditions (A) and (B) implies (P).
\end{proposition}

To characterize the relationship between passivity and the representations~\eqref{eq:ph-a} and~\eqref{eq:ph-b}, the following consequence of Taylor's theorem will be useful, see, e.g.~\cite[Section 4.5]{zeidler91-applied}.

\begin{lemma}\label{thm:f-F}
    Let $f\colon \RR^n \to \RR^n$ satisfy $f(0) = 0$ and $f\in C^k$ for some $k\geq 1$.
    Then $f(z) = F(z) z$ for some $F\colon \RR^n \to \RR^{n,n} \in C^{k-1}$.
    One such~$F$ is 
    $
        F(z) = \int_{0}^{1} Df(sz) \,\mathrm{d}s.
    $
\end{lemma}

\subsection{Towards (B)}\label{subsec:towards-B}
We will start by investigating the structure~\eqref{eq:ph-b}.
Under the assumption of injectivity of $\eta$, this structure turns out to be a representation of passive systems that explicitly incorporates energy features in the equations.

\begin{proposition}[(P) \implies (B)]\label{thm:P-B}
    Assume that (P) holds.
    Then the following are equivalent:
    \begin{enumerate}
        \item \label{item:P-B-propB}
            Property (B) holds.
        \item \label{item:P-B-eta-f}
            For all points $z_1, z_2 \in \RR^n$ with $\eta(z_1)=\eta(z_2)$ it holds that $f(z_1) = f(z_2)$.
        \item \label{item:P-B-m}
            There exists a map $m\colon \eta(\RR^n) \to \RR^n$ such that $f(z) = m(\eta(z))$ for all $z\in\RR^n$.
    \end{enumerate}
    Further, the following condition implies (B):
    \begin{itemize}
    \item
        The map $\eta$ is injective.
    \end{itemize}
\end{proposition}
\begin{proof}
    To show \ref{item:P-B-propB} \implies \ref{item:P-B-eta-f}, note that by the definition of (B) there exist functions $j,r \colon \eta(\RR^n) \to \RR^n$ such that $f(z) = j(\eta(z)) - r(\eta(z))$ for all $z \in \RR^n$.
    Hence, for $z_1, z_2 \in \RR^n$ with $\eta(z_1) = \eta(z_2)$ it follows that $f(z_1) = j(\eta(z_1)) - r(\eta(z_1)) = j(\eta(z_2)) - r(\eta(z_2)) = f(z_2)$.

    For \ref{item:P-B-eta-f} \implies \ref{item:P-B-m}, let us define the equivalence relation $\sim$ by 
    \begin{equation*}
        z \sim y :\Leftrightarrow \eta(z) = \eta(y)
    \end{equation*}
    and denote the equivalence class of $z \in \RR^n$ by $[z] \in \RR^n/\!\sim \,$.
    Then the map $\hat{\eta}\colon \RR^n/\!\sim \,\to \eta(\RR^n),~ [z] \mapsto \eta(z)$ is bijective and we can for $v\in \eta(\RR^n)$ define 
    \begin{equation*}
        m(v) \coloneq f\big(\psi(\hat{\eta}\inv(v))\big),
    \end{equation*}
    where $\psi\colon \RR^n/\!\sim \,\to \RR^n,~ [z] \mapsto z$ picks an arbitrary representative.
    Note that condition \ref{item:P-B-eta-f} ensures that the map $m\colon \eta(\RR^n) \to \RR^n$ is well defined, and that by definition we have $m(\eta(z)) = f(z)$ for all $z\in\RR^n$.

    For \ref{item:P-B-m} \implies \ref{item:P-B-propB}, observe that by property (P) we have
    \begin{equation*}
        \eta(z)\tp m(\eta(z)) = \eta(z)\tp f(z) = - \ell(z)\tp \ell(z) \leq 0,
    \end{equation*}
    or in other words $v\tp m(v) \leq 0$ for all $v \in \eta(\RR^n)$.
    Hence, we obtain a representation of~\eqref{eq:nonlinear-system} in the form~\eqref{eq:ph-b} by setting $j\coloneq0$ and $r\coloneq-m$.

    To finish the proof, note that the sufficient condition implies~\ref{item:P-B-eta-f} and hence also property (B).
\end{proof}

Note that ``\ref{item:P-B-eta-f})$\Rightarrow$\ref{item:P-B-propB})'' from \Cref{thm:P-B} can be interpreted as a generalization of~\cite[Theorem 7.1]{vanderschaft14-port} to nonlinear systems.
Furthermore, notice that under the assumptions of the inverse function rule it follows that $m$ from \Cref{thm:P-B} is as smooth as $f$ and $\eta$.

\begin{remark}\label{rem:P-B-insight-injectivity}
    Provided that $\eta$ is injective, a representation of~\eqref{eq:nonlinear-system} in the form~\eqref{eq:ph-b} is easy to obtain by setting $r(v) \coloneq -f(\eta\inv(v))$ for all $v \in \eta(\RR^n)$ and $j\coloneq 0$.
    Here, $\eta\inv$ denotes the inverse of $\eta\colon \RR^n \to \eta(\RR^n)$.
    Then 
    \begin{equation*}
        \eta(z)\tp r(\eta(z)) = - \eta(z)\tp f(z) = \ell(z)\tp \ell(z) \geq 0,
    \end{equation*}
    so that the system~\eqref{eq:nonlinear-system} is in the structure~\eqref{eq:ph-b}.
\end{remark}

Similar strategies can be used to investigate the relationship between~\eqref{eq:ph-a} and~\eqref{eq:ph-b}.

\begin{proposition}[(A) \implies (B)]\label{thm:A-B}
    Assume that (A) holds.
    Then the following are equivalent:
    \begin{enumerate}
    \item \label{item:A-B-propB}
        Property (B) holds.
    \item \label{item:A-B-eta-f}
        For  $\eta(z_1)=\eta(z_2)$ we have $\eta(z_1),\eta(z_2) \in \ker(J(z_1) - J(z_2) - R(z_1) + R(z_2))$.
    \item \label{item:A-B-m}
        There exists a map $m\colon \eta(\RR^n) \to \RR^n$ such that $(J(z) - R(z))\eta(z) = m(\eta(z))$.
    \end{enumerate}
    Further, any of the following conditions implies (B):
    \begin{itemize}
    \item
        The map $\eta$ is injective.
    \item
        $J$ and $R$ are constant.
    \end{itemize}
\end{proposition}
\begin{proof}
    To show \ref{item:A-B-propB} \implies \ref{item:A-B-eta-f}, note that by definition of (A) and (B) there exist functions $J, R \colon \RR^n \to \RR^{n,n}$ and $j,r \colon \eta(\RR^n) \to \RR^n$ such that $f(z) = (J(z) - R(z))\eta(z) = j(\eta(z)) - r(\eta(z))$ for all $z \in \RR^n$.
    Hence, for $z_1, z_2 \in \RR^n$ with $\eta(z_1) = \eta(z_2)$ it follows that $(J(z_1) - R(z_1))\eta(z_1) = j(\eta(z_1)) - r(\eta(z_1)) = j(\eta(z_2)) - r(\eta(z_2)) = (J(z_2) - R(z_2))\eta(z_2)$, implying $(J(z_1) - J(z_2) - R(z_1) + R(z_2)) \eta(z_1) = 0$.

    For \ref{item:A-B-eta-f} \implies \ref{item:A-B-m}, we can proceed as in the proof of \Cref{thm:P-B} to see that the map $\hat{\eta}\colon \RR^n/\!\sim \,\to \eta(\RR^n), ~ [z] \mapsto \eta(z)$ is bijective.
    Then, for $v\in \eta(\RR^n)$ we can define 
    \begin{equation*}
        m(v) \coloneq \Big(J\big(\psi(\hat{\eta}\inv(v))\big) - R\big(\psi(\hat{\eta}\inv(v))\big)\Big) v,
    \end{equation*}
    where $\psi\colon \RR^n/\!\sim \,\to \RR^n,~ [z] \mapsto z$ again picks an arbitrary representative.
    Now, observe that condition \ref{item:A-B-eta-f} ensures that the map $m\colon \eta(\RR^n) \to \RR^n$ is well defined and that by definition we have $m(\eta(z)) = ( J(z) - R(z) ) \eta(z)$ for all $z\in\RR^n$.

    For \ref{item:A-B-m} \implies \ref{item:A-B-propB}, observe that by $R(z) = R(z)\tp \succeq 0$ we have
    \begin{align*}
        \eta(z)\tp m(\eta(z)) 
        & = \eta(z)\tp (J(z) - R(z)) \eta(z) \\
        & = - \eta(z)\tp R(z) \eta(z) \leq 0,
    \end{align*}
    or in other words $v\tp m(v) \leq 0$ for all $v \in \eta(\RR^n)$.
    Hence, we obtain a representation of~\eqref{eq:nonlinear-system} in the form~\eqref{eq:ph-b} by setting $j\coloneq0$ and $r\coloneq-m$.

    To finish the proof, note that the first sufficient condition implies \ref{item:A-B-eta-f} and hence also property (B), and that with the second sufficient condition we can define $j(v) \coloneq Jv$ and $r(v) \coloneq Rv$ for all $v\in\eta(\RR^n)$, where the structural properties of $j$ and $r$ follow from the respective properties of $J$ and $R$.
\end{proof}

\begin{remark}\label{rem:A-B-insight-injectivity}
    To give a little insight into the sufficiency of the injectivity of~$\eta$ in \Cref{thm:A-B}, observe that we may also arrive at a representation in the form~\eqref{eq:ph-b} by setting 
    \begin{equation*}
        j(v) \coloneq J(\eta\inv(v)) v \und r(v) \coloneq R(\eta\inv(v)) v
    \end{equation*}
    for all $v \in \eta(\RR^n)$.
\end{remark}

\subsection{Towards (A)}\label{subsec:towards-A}

Unfortunately, unlike the situation in \Cref{subsec:towards-B}, the properties (P) and (B) are not particularly helpful in implications towards (A).
This is because if a decomposition $f(z) = (J(z) - R(z))\eta(z)$ with $J(z) = -J(z)\tp$ and $R(z) = R(z)\tp$ is known, the passivity condition $\eta(z)\tp f(z) = - \ell(z)\tp \ell(z)$ and the condition $v\tp r(v) \geq 0$ for $v \in \eta(\RR^n)$ only imply $v\tp R(z) v \geq 0$ for $v \in \Span\{\eta(z)\}$, in contrast to $R(z) \succeq 0$.
In other words, $R(z) \succeq 0$ is an intrinsic property in the model class \eqref{eq:ph-a} that does not follow from passivity features.
Note that this phenomenon is not unique to nonlinear systems:
If $f(z) = - K Q z$ with singular $Q = Q\tp \succeq 0$, then $z\tp Q K Q z \succeq 0$ only implies $v\tp K v \geq 0$ for $v \in \ran(Q) \neq \RR^n$.  

In \Cref{thm:P-A} we highlight that $R(z) \succeq 0$ is an intrinsic property of the structure~\eqref{eq:ph-a} and provide a sufficient condition for (A) to hold.
Note that the equivalence of~\ref{item:P-A-propA}) and~\ref{item:P-A-M}) is well known and is mentioned in, e.g.,~\cite{cheng05-feedback}.
See also \Cref{rem:P-A-mclachlan} for comments on the sufficient condition, which has appeared similarly in~\cite{mclachlan99-geometric}.

\begin{proposition}\label{thm:P-A}
    The following are equivalent:
    \begin{enumerate}
    \item \label{item:P-A-propA}
        Property (A) holds.
    \item \label{item:P-A-M}
        There exists a matrix-valued map $M\colon \RR^n \to \RR^{n,n}$ such that $f(z) = M(z) \eta(z)$ and $M(z) \preceq 0$ for all $z\in \RR^n$.
    \end{enumerate}
    Further, the following condition implies (A):
    \begin{itemize}
    \item
        It holds that $f, \eta \in C^1$, $\eta$ is bijective, $(f\circ\eta\inv)(0) = 0$, $D\eta(z)$ is invertible for all $z\in\RR^n$, and $Df(z) D\eta(z)\inv \preceq 0$ for all $z\in\RR^n$.
    \end{itemize}
\end{proposition}
\begin{proof}
    The equivalence \ref{item:P-A-propA} \!\!$\Leftrightarrow$\! \ref{item:P-A-M} is straightforward with $J(z) = M(z)\skew$ and $R(z) = -M(z)\sym$.

    Regarding the sufficient condition, first observe that the inverse function rule states that $\eta\inv$ is differentiable with derivative $D\eta\inv(\eta(z)) = D\eta(z)\inv$ for all $z\in\RR^n$.
    Now notice that $D(f\circ \eta\inv)(v) = Df(\eta\inv(v)) D\eta\inv(v) = Df(z) D\eta(z)\inv$ for $v=\eta(z)$ due to the chain rule, which shows that the Jacobian of $f\circ \eta\inv$ is pointwise negative semidefinite.
    Hence \Cref{thm:f-F} implies the existence of $N\colon \RR^n \to \RR^{n,n}$ such that 
    \begin{equation*}
       (f\circ \eta\inv )(v) = N(v) v
    \end{equation*}
    with $N(v) \preceq 0$ for all $v\in\RR^n$.
    Plugging in $v = \eta(z)$, we obtain $f(z) = N(\eta(z)) \eta(z)$ for all $z\in\RR^n$.
    Hence, setting $M(z) \coloneq N(\eta(z))$ and using the equivalence of \ref{item:P-A-M} and \ref{item:P-A-propA} finishes the proof.
\end{proof}

\begin{remark}\label{rem:P-A-mclachlan}
    A similar strategy as in the sufficient condition of \Cref{thm:P-A} was used in the proofs of Propositions 2.4 and 2.11 in~\cite{mclachlan99-geometric}.
    The differences between~\cite{mclachlan99-geometric} and our approach are as follows.
    First,~\cite{mclachlan99-geometric} makes local statements around nondegenerate critical points $\bar{z}\in\RR^n$ of $\ham$, which are points where $\eta(\bar{z}) = 0$ and $D\eta(\bar{z})$ is invertible.
    They then use the fact that locally around $\bar{z}$ there exists a coordinate transformation $\phi$ such that $\ham(\phi(z)) = \frac12 \phi(z) \tp B \phi(z)$ with $B$ being nonsingular and subsequently only consider the case $\eta(z) = Bz$.
    In contrast, our statement is of a global nature and therefore requires $D\eta(z)$ to be globally invertible. 
    Furthermore, we make the change of variables explicit.
    The second difference is in the semidefiniteness assumption $Df(z) D\eta(z)\inv \preceq 0$.
    In~\cite[Propositions 2.4 and 2.11]{mclachlan99-geometric}, the authors only consider the cases of skew-symmetric or positive definite $Df(z) D\eta(z)\inv$.%
    \footnote{The negative definite case can be recovered by replacing $f$ by $-f$. Further, the semidefinite case is covered in~\cite[Proposition 2.10]{mclachlan99-geometric} using a different strategy, which we discuss in \Cref{subsec:general-case}.}
    For the latter, they replace our assumption by the assumption that locally $\eta(z)\tp f(z) \geq b \norm{f(z)} \norm{\eta(z)}$ for a constant $b>0$.
    Together with the invertibility of $Df(z)$, this then implies the definiteness of $N$ as in the proof of \Cref{thm:P-A}.
    Hence, our assumption is generally weaker, but might be more difficult to verify.
\end{remark}

\begin{remark}
    In the case of linear time-invariant systems, the assumptions in the sufficient condition of \Cref{thm:P-A} read as follows.
    Assume $\dot{z} = Az$ for some $A \in \RR^n$, and assume $\eta(z) = Q z$ with $Q = Q\tp \succeq 0$.
    Then $\eta$ is bijective if and only if $Q$ is invertible. 
    In this case, we have $D\eta(z) = Q$ as well as $Df(z) D\eta(z)\inv = AQ\inv$.
    Hence, the condition $Df(z) D\eta(z)\inv \preceq 0$ reduces to $AQ\inv \preceq 0$, which is equivalent to $AQ\inv + Q\inv A\tp \preceq 0$.
    Multiplying this inequality by $Q=Q\tp$ from both sides gives 
    \begin{equation*}
        QA + A\tp Q \preceq 0,
    \end{equation*}
    which is well known and can also be found in, e.g.,~\cite{willems72-dissipative2}.
\end{remark}

Let us remark on how the sufficient condition in \Cref{thm:P-A} can be used  to construct \pH realizations, which is the focus of \Cref{sec:constructions}.

\begin{remark}\label{rem:P-A-definition-of-J-R}
    From \Cref{thm:f-F} it follows that a possible choice for $M(z)$ in the proof of the sufficient condition of \Cref{thm:P-A} is
    \begin{equation*}
        M(z) = \int_{0}^{1} Df(sz) D\eta(sz)\inv \ds.
    \end{equation*}
    In particular, we may arrive at a representation of~\eqref{eq:nonlinear-system} in the form~\eqref{eq:ph-a} by choosing $J(z) \coloneq M(z)\skew$ and $R(z) \coloneq - M(z)\sym \preceq 0$. 
\end{remark}

\begin{remark}\label{rem:P-A-and-A-B}
    If the sufficient condition in \Cref{thm:P-A} holds, then we can write the system in the form~\eqref{eq:ph-b}, since $\eta$ is assumed to be bijective.
    One possibility is setting 
    \begin{equation*}
        j(v)\coloneq 0 \und r(v) \coloneq - N(v) v
    \end{equation*} 
    for all $v\in \RR^n$, where $N$ is defined as in the proof of \Cref{thm:P-A}.
\end{remark}

We obtain a similar result for the relationship of~\eqref{eq:ph-b} and~\eqref{eq:ph-a}.

\begin{proposition}[(B) \implies (A)]\label{thm:B-A}
    If property (B) holds, then any of the following conditions implies~(A):
    \begin{itemize}
    \item
        There exists a matrix-valued map $N\colon \RR^n \to \RR^{n,n}$ such that $j(v) - r(v) = N(v) v$ and $N(v) \preceq 0$ for all $v \in \eta(\RR^n)$.
    \item
        We have $j(0) - r(0) = 0$, $j-r \in C^1$, and $D(j-r)(v) \preceq 0$ for all $v \in \eta(\RR^n)$.
    \end{itemize}
\end{proposition}
\begin{proof} 
    Regarding the first sufficient condition, note that we can define $J(z) \coloneq N(\eta(z))\skew$ and $R(z) \coloneq -N(\eta(z))\sym$ such that
    \begin{equation*}
        ( J(z) - R(z) )\eta(z) = N(\eta(z)) z = j(\eta(z)) - r(\eta(z)),
    \end{equation*}
    where $R(z) \succeq 0$ follows from $N(\eta(z)) \preceq 0$.

    For the second sufficient condition, observe that \Cref{thm:f-F} implies the existence of a map $N$ as in the first sufficient condition.
\end{proof}

Unfortunately, the sufficient condition in \Cref{thm:P-A} can be quite restrictive, see \Cref{ex:rigid-body} in \Cref{sec:examples}.
A less restrictive sufficient condition is obtained in the following result.

\begin{proposition}\label{thm:P-A-fix}
    The following condition implies (A):
    \begin{itemize}
    \item
        It holds that $f, \eta \in C^1$, $\eta$ is bijective, $(f\circ\eta\inv)(0) = 0$, $D\eta(z)$ is invertible for all $z\in\RR^n$, and there exists $P\colon \RR^n\to\RR^{n,n}$ such that 
        \begin{equation}\label{eq:P-pointwise}
            M(z) + P(z) \preceq 0 \und P(z) \eta(z) = 0
        \end{equation}
        for all $z \in \RR^n$, where $M(z) = \int_{0}^{1} Df(sz) D\eta(sz)\inv \ds$ is defined as in \Cref{rem:P-A-definition-of-J-R}.
    \end{itemize}
\end{proposition}
\begin{proof}
    The proof of \Cref{thm:P-A} shows that our assumptions imply the well-definedness of $M(z)$ and $f(z) = M(z) \eta(z)$ for all $z\in \RR^n$.
    To show the claim, note that we can choose $J(z) = M(z)\skew + P(z)\skew$ and $R(z) = - M(z)\sym - P(z)\sym$ to arrive at a representation in the form~\eqref{eq:ph-a}.
\end{proof}

\begin{remark}\label{rem:P-A-fix-phi}
    One possible choice for the function $P$ in \Cref{thm:P-A-fix} is $P(z) = \int_{0}^{1} \phi(sz) \ds$, where $\phi \colon \RR^n \to \RR^{n,n}$ is such that 
    \begin{equation*}
        Df(z) D\eta(z)\inv + \phi(z) \preceq 0 \und \phi(sz) \eta(z) = 0
    \end{equation*} 
    for all $z\in\RR^n$ and $s\in[0,1]$.
\end{remark}

In the next section, we shift our attention towards finding a function $P$ as in \Cref{thm:P-A-fix}.

\section{Constructing \portHamiltonian representations}
\label{sec:constructions}

The results from \Cref{sec:characterizations} can be used to construct \portHamiltonian representations of passive systems.
Throughout this section, we focus on representations of the form~\eqref{eq:ph-a} and make the following assumption.

\begin{assumption}\label{as:constructions}
    Property (P) holds, and we have $f,\eta \in C^1$, $\eta$ is bijective, $(f\circ \eta\inv)(0) = 0$, and $D\eta(z)$ is invertible for all $z \in \RR^n$.
\end{assumption}

As we have mentioned before, the sufficient condition in \Cref{thm:P-A} can be restrictive, and thus also the construction in \Cref{rem:P-A-definition-of-J-R} is not always feasible.
In particular, if the condition $Df(z)D\eta(z)\inv \preceq 0$ from \Cref{thm:P-A} does not hold, then the construction in the proof can still be carried out, but gives $f(z) = M(z)\eta(z)$ with $M(z)\not\preceq 0$ in general.
Luckily, \Cref{thm:P-A-fix} provides a different strategy to construct \portHamiltonian representations in these cases.
If we can find $P$ as in~\eqref{eq:P-pointwise}, then we can again construct a representation of the system in the form~\eqref{eq:ph-a} by setting $J(z) \coloneq M(z)\skew + P(z)\skew$ and $R(z) := - M(z)\sym - P(z)\sym$. 
The task of finding a suitable function $P$ is not trivial in general, which is why we first consider a special case.
For ease of notation, we mostly suppress the state dependency in the following.

\subsection{Conservative systems}
\label{subsec:conservative-systems}

Let us assume that $f = J \eta$ for some unknown $J = -J\tp$.
The problem of identifying~$J$ in this conservative case was studied in, e.g.,~\cite{quispel96-solving}, where a full characterization of possible functions $J$ was given.
Here, we present an alternative characterization.

We aim for $P$ such that 
\begin{equation}\label{eq:P-conservative}
    P \eta = 0 \und M\sym + P\sym = 0,
\end{equation}
from which we immediately deduce that $P\sym = - M\sym$ and thus $P\skew \eta = M\sym \eta$.
Since $P\skew$ is skew-symmetric, it is determined by its $\frac{n^2 - n}{2}$ entries in the upper triangular part (without the diagonal), and thus $P\skew \eta = M\sym \eta$ is a linear system of equations in the entries of $P\skew$, where we have $n$ equations for $\frac{n^2 - n}{2}$ unknowns.
A solution of~\eqref{eq:P-conservative} is $P = - M + J$, but we can not expect the solution to be unique.
Therefore, in the following we restrict ourselves to the case that~$P\skew$ is tridiagonal.
As we will see, this is possible for states $z\in\RR^n$ with $\eta_i(z) \neq 0$ for $i=2,\dots,n-1$.
Let us collect the entries of the superdiagonal of $P\skew$ in $\pj \coloneq \vec{p_1 & \cdots & p_{n-1}}$.
Then $P\skew \eta = M\sym \eta$ may be written as 
\begin{equation}\label{eq:T-smallsystem-conservative}
    \Tj(\eta) \pj = M\sym \eta
\end{equation} 
with
\begin{equation*}
    \Tj(\eta) \coloneq 
    \vec{
        \eta_2  \\
        - \eta_1 & \eta_3  \\
        & \ddots & \ddots  \\
        & & -\eta_{n-2} & \eta_{n} \\
        & & &  -\eta_{n-1}
    } \in \RR^{n, n-1}.
\end{equation*}
If $\eta_i \neq 0$ for $i=2,\dots,n-1$, then $\rank(\Tj(\eta)) = n-1$.
In particular, the system~\eqref{eq:T-smallsystem-conservative} has a unique solution in this case, and there exists a unique $P$ with $P\eta = 0,~ P\sym + M\sym = 0$, where $P\skew$ is tridiagonal.
We summarize our findings in the following theorem.

\begin{theorem}[conservative systems]\label{thm:conservative-systems}
    In addition to \Cref{as:constructions}, assume $f = J \eta$ for some unknown $J = -J\tp$.
    Then there exists $P$ such that~\eqref{eq:P-conservative} holds, and for states $z$ with $\eta_i(z) \neq 0$ for $i = 2, \dots, n-1$ we can choose $P\skew$ as a uniquely determined tridiagonal matrix.
\end{theorem}

\begin{remark}\label{rem:P-continuity-conservative}
    An obvious question is the continuity of $P$ with respect to the state variable $z$.
    We focus on the special case that $P\skew$ can be chosen as a tridiagonal matrix.
    From the discussion above, we know that on the dense open subset $E \coloneq \{ \eta(z)\in\RR^n ~|~ z\in\RR^n \text{ with } \eta_i(z) \neq 0 ~ \text{for $i = 2,\dots,n-1$}\}$, the system~\eqref{eq:T-smallsystem-conservative} has a unique solution $\pj$.
    Since $\rank(\Tj(\eta)) = n-1$ on $E$, this $\pj$ is the unique solution to the normal equations
    \begin{equation*}
        \Tj(\eta)\tp \Tj(\eta) \pj = \Tj(\eta)\tp M\sym \eta,
    \end{equation*}
    or, in other words, $\pj = (\Tj(\eta)\tp \Tj(\eta))\inv \Tj(\eta)\tp M\sym \eta$.
    As $\eta\inv$ is continuous by our assumptions, on $E$ the solution $\pj$ depends continuously on $z$.
    Hence, if 
    \begin{equation*}
        \eta \mapsto (\Tj(\eta)\tp \Tj(\eta))\inv \Tj(\eta)\tp M\sym \eta
    \end{equation*}
    extends continuously from $E$ to $\RR^n = \eta(\RR^n)$, then $\pj$ (and therefore also $P$) is continuous with respect to $z$.
    In~\Cref{ex:rigid-body}, this continuous extension is possible.

    Further, we remark that the case of analytic $f$ and $\eta$ has been studied in~\cite[Proposition 2.4]{mclachlan99-geometric}.
\end{remark}

\subsection{The general case}
\label{subsec:general-case}

As we have mentioned in the introduction, the construction of $J=-J\tp$ and $R=R\tp\succeq0$ such that $f = (J - R)\eta$ was studied in, e.g.,~\cite{mclachlan99-geometric,ortega02-interconnection,wang03-generalized}.
A common feature in these approaches is that~$f$ is decomposed into $f = f_1 + f_2$, where~$f_1$ and~$f_2$ correspond to the energy conserving and energy dissipating parts of the dynamics, respectively.
Naturally, the function~$J$ is then constructed from~$f_1$, and~$R$ is constructed from~$f_2$.
In all of these approaches,~$R$ contains the factor~$\norm{\eta}^{-2}$ stemming from the fact that if $f = v + \beta \eta$ with $v \in \Span\{\eta\}^\perp$ then $\beta = \norm{\eta}^{-2} f\tp \eta$.
Here, we present an approach where $R$ does not necessarily contain this factor.
The idea in our approach is to use information from $f = M\eta$ even though $M \not\preceq 0$ in general.

In the case $\eta=0$, the matrix $P$ can be chosen arbitrarily as long as $M\sym+P\sym\preceq 0$, e.g.,~$P = - M\sym$.
Hence, we will restrict our analysis to the case $\eta\neq 0$.
Let $\subs \subseteq \RR^n$ be a subspace with $\RR^n = \Span\{\eta\} \oplus \subs$.
It is clear that in this case $\dim(\subs) = n-1$.
Let us write $\subs^\perp = \Span\{w\}$ for some $w\in\RR^n$ and define the projection $\proj \coloneq \frac{\eta w\tp}{w\tp \eta}$ which projects onto $\Span\{\eta\}$ along the subspace $\subs$. 
We observe that $\proj$ is well-defined, since $w\tp \eta = 0$ would imply $\eta \in (\subs^\perp)^\perp=\subs$ which contradicts our assumptions $\eta\neq 0$ and $\RR^n = \Span\{\eta\} \oplus \subs$.
Additionally, $\cproj \coloneq I_n-\proj$ is again a projection which projects onto $\subs$ along $\Span\{\eta\}$.
In particular, we have $\cproj \eta = 0$.
We can now decompose $M\sym$ as         
\begin{align*}
    M\sym 
    & = 
    (\proj + \cproj)\tp M\sym (\proj + \cproj) 
    \\
    & 
    = 
    \proj\tp M\sym \proj
    +
    \proj\tp M\sym \cproj
    +
    \cproj\tp M\sym \proj
    +
    \cproj\tp M\sym \cproj.
\end{align*}
The first summand is negative semidefinite, since $\ran(\proj) = \Span\{\eta\}$ and $(\alpha\eta)\tp M\sym (\alpha \eta)=\alpha^2\eta\tp f\le 0$ for all $\alpha\in\RR$. 
This motivates the choice 
\begin{equation}\label{eq:Psym-general}
    \begin{aligned}
        P\sym 
        & = 
        -
        \proj\tp M\sym \cproj
        -
        \cproj\tp M\sym \proj
        -
        \cproj\tp M\sym \cproj
        \\
        & = 
        - M\sym + \proj\tp M\sym \proj.
    \end{aligned}
\end{equation}
In general, $P\sym \eta \neq 0$, which is why we need to determine $P\skew$ such that $P\skew \eta = - P\sym \eta$.
The marix $P\skew$ can be constructed similarly as in \Cref{subsec:conservative-systems}. 
Here we need to solve the system 
\begin{equation}\label{eq:T-system-general}
    \bfTj(\eta) \bfpj = -P\sym \eta = \cproj\tp M\sym \eta,
\end{equation}
where $\bfpj$ lexicographically orders the entries of the strict upper triangular part of $P\skew$.
Since $\eta\neq0$, from the discussion in \Cref{subsec:conservative-systems} it follows that $\rank(\bfTj(\eta)) = n-1$. 
Further, $\eta\tp \bfTj(\eta)=0$ so that $\ran(\bfTj(\eta)) = \Span\{\eta\}^\perp$. 
System \eqref{eq:T-system-general} now has a solution since $\cproj \eta = 0$ so that
$
    \eta\tp \cproj\tp M\sym \eta = (\cproj\eta)\tp M\sym \eta = 0
$
and hence $\cproj\tp M\sym \eta\in \ran(\bfTj(\eta))$. 
Note that the arguments concerning the choice of~$P\skew$ as a tridiagonal matrix can be adapted to the present setting.

So far we have considered an arbitrary subspace $\subs$ satisfying $\RR^n = \Span\{\eta\} \oplus \subs$.
Here, we want to further remark on two particular choices for $\subs$.
The canonical choice $\subs = \Span\{\eta\}^\perp$ guarantees that the projections $\proj$ and $\cproj$ are well defined.
In this case we can choose $w = \eta$ so that $\proj = \frac{\eta\eta\tp}{\norm{\eta}^2}$ and
\begin{equation*}
    M\sym + P\sym
    =
    \frac{\eta\eta\tp M\sym \eta \eta\tp}{\norm{\eta}^4}
    =
    \frac{\eta \eta\tp f\tp \eta}{\norm{\eta}^4}.
\end{equation*}
This is the construction from~\cite{ortega02-interconnection}.
If $\eta\tp M\sym \eta \neq 0$, then $\Span\{\eta\}\not\subseteq \Span\{M\sym\eta\}^\perp$ and another interesting choice is $\subs = \Span\{M\sym\eta\}^\perp$. 
For this choice, the mixed terms in the first line of \eqref{eq:Psym-general} vanish, leaving us with $P\sym = -\cproj^T M\sym \cproj$. 
Thus $P\sym \eta = 0$ and no skew-symmetric matrix $P\skew$ needs to be constructed. 
In this case $M\sym + P\sym$ reads as 
\begin{equation*}
    M\sym + P\sym
    =
    \frac{M\sym \eta\eta\tp M\sym}{\eta\tp M\sym \eta}.
\end{equation*}
We summarize our findings in the following theorem.

\begin{theorem}\label{thm:P-general}
    Let \Cref{as:constructions} hold, and assume $\eta \neq 0$.
    Then there exists $P$ such that~\eqref{eq:P-pointwise} holds.
    For some subspace $\subs \subseteq \RR^n$ with $\RR^n=\Span\{\eta\}\oplus \subs$, one possible choice is $P = P\sym + P\skew$, where $P\sym$ is chosen as in~\eqref{eq:Psym-general} and the entries of $P\skew$ are determined by a solution of the system \eqref{eq:T-system-general}.
    If $\eta\tp M\sym \eta \neq 0$, then we can choose  $P= -\cproj^T M\sym \cproj=-M\sym + \frac{M\sym \eta\eta\tp M\sym}{\eta\tp M\sym\eta}$.
\end{theorem}

\begin{remark}\label{rem:singularities}
    As we have mentioned in the introduction, singularities at $\eta = 0$ are drawbacks of the approaches presented in~\cite{ortega02-interconnection,wang03-generalized}.
    In the examples in \Cref{sec:examples}, the approach from \Cref{rem:P-A-definition-of-J-R} does not lead to these singularities.
    However, as we have seen in the discussion above, if $M \not\preceq 0$ and $P$ as in~\eqref{eq:P-pointwise} is needed to construct \portHamiltonian representations, then new singularities can occur.
\end{remark}

\begin{remark}\label{rem:manifold-turnpikes}
    For $f = M\eta = (J-R)\eta$ with $J=-J\tp$ and $R=R\tp$ we have $\eta\tp f = \eta\tp M\sym \eta = - \eta\tp R \eta$, and in particular $\MM \coloneq \{ z \in \RR^n ~|~ \eta\tp M\sym \eta = 0\} = \{z \in \RR^n ~|~ \eta\tp R \eta = 0\}$.
    Under suitable assumptions on the system dynamics, it can be shown that the set $\MM$ is a smooth submanifold of $\RR^n$, and that the trajectory $z^*$ minimizing the supplied energy $\int_{0}^{T} y\tp u \dt$ to~\eqref{eq:ph-a} spends most of its time close to $\MM$. 
    We refer the interested reader to~\cite{karsai24-manifold} for details.
\end{remark}

\section{Examples}\label{sec:examples}

Let us illustrate some of the constructions from \Cref{sec:characterizations,sec:constructions} using examples.
We begin with an important special case, where the nonlinearity stems from the gradient of the Hamiltonian.
Note that this example also covers linear \pH systems with Hamiltonian $\ham(z) = \tfrac12 z\tp Q z$, $Q = Q\tp \succ 0$.

\begin{example}[constant $J$ and $R$]\label{ex:nonlinear-eta}
    Consider a system of the form 
    \begin{equation}\label{eq:nonlinear-eta}
        \dot{z}
        =
        f(z)
        =
        (J - R)\eta(z)
        =
        K \eta(z),
    \end{equation}
    where $J = -J\tp$, $R = R\tp \succeq 0$ are possibly unknown.
    If we assume that \Cref{as:constructions} holds, then
    \begin{equation*}
        Df(z) D\eta(z)\inv
        =
        K D\eta(z) D\eta(z)\inv = K,
    \end{equation*}
    so that the construction from \Cref{rem:P-A-definition-of-J-R} recovers $J = K\skew$ and $R = - K\sym$.
\end{example}

For the second example, we consider a rigid body in three spatial dimensions spinning around its center of mass in the absence of gravity, see~\cite[Examples 4.2.4, 6.2.1]{vanderschaft17-l2gain}.

\begin{example}[spinning rigid body]\label{ex:rigid-body}
    Consider the system
    \begin{equation}\label{eq:spinning-rigid-body}
        \begin{aligned}
            \vec{\dot{z}_1 \\ \dot{z}_2 \\ \dot{z}_3}
            = 
            \vec{
            0 & - z_3 & z_2 \\
            z_3 & 0 & - z_1 \\
            -z_2 & z_1 & 0
            }
            \vec{
            \frac{z_1}{I_1} \\
            \frac{z_2}{I_2} \\
            \frac{z_3}{I_3}
            }
            =
            \vec{
            z_2 z_3 (\frac{1}{I_3} - \frac{1}{I_2}) \\
            z_1 z_3 (\frac{1}{I_1} - \frac{1}{I_3}) \\
            z_1 z_2 (\frac{1}{I_2} - \frac{1}{I_1})
            }
            = f(z),
        \end{aligned}
    \end{equation}
    where the state of the system is the vector of angular momenta $z = (z_1,z_2,z_3)$ in the three spatial dimensions, and the Hamiltonian of the system is given by $\ham(z) = \tfrac12 \big( \tfrac{z_1^2}{I_1} + \tfrac{z_2^2}{I_2} + \tfrac{z_3^2}{I_3} \big)$.
    Here, $I_1,I_2,I_3$ are the principal moments of inertia.
    Note that~\eqref{eq:spinning-rigid-body} is in the form~\eqref{eq:ph-a} with $J(z) = \Bigtextvec{0 & - z_3 & z_2 \\ z_3 & 0 & - z_1 \\ -z_2 & z_1 & 0}.$
    We obtain 
    \begin{align*}
        S(z) & 
        \coloneq Df(z) D\eta(z)\inv 
        =
        \vec{
        0 & - z_3 (1 - \frac{I_2}{I_3}) & z_2 ( 1 - \frac{I_3}{I_2}) \\
        z_3 (1 - \frac{I_1}{I_3}) & 0 & - z_1 ( 1 - \frac{I_3}{I_1}) \\
        - z_2 (1 - \frac{I_1}{I_2}) & z_1 ( 1 - \frac{I_2}{I_1}) & 0 \\
        }
    \end{align*}
    and 
    \begin{equation*}
        S(z) + S(z)\tp 
        =
        \vec{
        0 & z_3 (\frac{I_2 - I_1}{I_3}) & z_2 (\frac{I_1 - I_3}{I_2}) \\
        z_3 (\frac{I_2 - I_1}{I_3}) & 0 & z_1 (\frac{I_3 - I_2}{I_1}) \\
        z_2 (\frac{I_1 - I_3}{I_2}) & z_1 (\frac{I_3 - I_2}{I_1}) & 0
        },
    \end{equation*}
    which is in general indefinite since the upper $2\times2$ block has the structure $[\begin{smallmatrix} 0 & a \\ a & 0\end{smallmatrix}]$.
    In particular, this example shows that the sufficient condition in \Cref{thm:P-A} is not necessary.
    Integration of $S(sz)$ yields
    \begin{equation*}
        M(z) 
        = 
        \int_{0}^{1} S(sz) \ds 
        =
        \frac12 
        \vec{
        0 & - z_3 (1 - \frac{I_2}{I_3}) & z_2 ( 1 - \frac{I_3}{I_2}) \\
        z_3 (1 - \frac{I_1}{I_3}) & 0 & - z_1 ( 1 - \frac{I_3}{I_1}) \\
        - z_2 (1 - \frac{I_1}{I_2}) & z_1 ( 1 - \frac{I_2}{I_1}) & 0
        },
    \end{equation*}
    which is again indefinite.

    Since we know that the system is conservative, we can use the ideas from \Cref{subsec:conservative-systems} to find a matrix $P(z)$ such that $P(z) \eta(z) = 0$ and $P(z)\sym + M(z)\sym = 0$.
    The ansatz of tridiagonal $P(z)\skew$ leads to the system 
    \begin{equation*}
        T(z) p(z) 
        = 
        \vec{
            \frac{z_{2}}{I_{2}} & 0 \\
            - \frac{z_{1}}{I_{1}} & \frac{z_{3}}{I_{3}} \\
            0 & - \frac{z_{2}}{I_{2}}
        }
        \vec{p_1 \\ p_2}
        =
        M(z)\sym \eta(z)
    \end{equation*}
    from which we obtain the solution 
    \begin{equation*}
        P(z)\skew = 
        \frac14
        \vec{
            0 & \frac{z_{3} (I_{2} - I_{3})}{I_{3}} & 0 \\
            - \frac{z_{3} (I_{2} - I_{3})}{I_{3}} & 0 & \frac{z_{1} (I_{2} - I_{1} )}{I_{1}} \\
            0 & - \frac{z_{1} (I_{2} - I_{1})}{I_{1}} & 0
        }.
    \end{equation*}
    For this $P(z)\skew$, we obtain 
    \begin{equation*}
        M(z) + P(z) 
        =
        \frac14
        \vec{
            0 & \frac{z_{3} \left(I_{1} + 2 I_{2} - 3 I_{3}\right)}{I_{3}} & - \frac{z_{2} \left(I_{1} - 2 I_{2} + I_{3}\right)}{I_{2}}
            \\
            - \frac{z_{3} \left(I_{1} + 2 I_{2} - 3 I_{3}\right)}{I_{3}} & 0 & - \frac{z_{1} \left(3 I_{1} - 2 I_{2} - I_{3}\right)}{I_{1}}
            \\
            \frac{z_{2} \left(I_{1} - 2 I_{2} + I_{3}\right)}{I_{2}} & \frac{z_{1} \left(3 I_{1} - 2 I_{2} - I_{3}\right)}{I_{1}} & 0
        }.
    \end{equation*}
    By construction we have $(M(z) + P(z))\eta(z) = M(z) \eta(z)$ and $M(z)\sym + P(z)\sym = 0$.
    In particular, $M(z) + P(z)$ is a suitable choice for the \pH representation of~\eqref{eq:spinning-rigid-body} that differs from the usual choice for this example.
\end{example}
Although we have only considered finite dimensional systems so far, let us illustrate that the methods from \Cref{sec:characterizations,sec:constructions} may potentially be used for infinite dimensional systems as well. 
In the infinite dimensional setting, the transposes in the finite dimensional definitions of $M(z)\skew$ and $M(z)\sym$ are replaced by formal adjoints.
As we do not include a rigorous discussion about the domains of the operators, the following derivations should be understood on a formal level.
\begin{example}[quasilinear wave equation]\label{ex:damped-wave}
    \newcommand{\visc}[1]{#1} 
    As in~\cite{giesselmann24-energy}, let us consider
    \begin{equation*}
        \begin{aligned}
            \partial_t \rho & = - \partial_x v, \\
            \partial_t v & = - \partial_x p(\rho) - \gamma F(v) \visc{+ \nu \partial_x^2 v}
        \end{aligned}
    \end{equation*}
    on $\Omega = [0,\ell]$ together with the boundary conditions $p(\rho(\cdot,0)) \visc{- \nu \partial_x v(\cdot, 0)} = p_0$, $p(\rho(\cdot,\ell)) \visc{- \nu \partial_x v(\cdot, \ell)} = p_\ell$ and initial conditions $(\rho,v)(0,\cdot) = (\rho_0, v_0)$ in $\Omega$.
    Here, the term $\gamma F(v)$ with $\gamma \geq 0$ models friction forces, and we assume that $F\in C^1$ is odd with $F(v) \geq 0$ for $v \geq 0$.
    \visc{Similarly, the term $\nu \partial_x^2 v$ with $\nu \geq 0$ models viscous forces.}
    For $P(\rho)$ such that $P'(\rho) = p(\rho)$, the associated Hamiltonian reads as $\ham(\rho, v) = \int_{0}^{\ell} P(\rho) + \tfrac12 v^2 \dx$
    with
    \begin{equation*}
        \eta(\rho, v) = \ham'(\rho, v) = \vec{p(\rho) \\ v}.
    \end{equation*}
    Setting $z \coloneq (z_1, z_2) \coloneq (\rho, v)$, we obtain 
    \begin{equation*}
        \partial_t z 
        = 
        \vec{ \partial_t z_1 \\ \partial_t z_2}
        =
        \vec{ - \partial_x z_2 \\ - \partial_x p(z_1) - \gamma F(z_2) \visc{+ \nu \partial_x^2 z_2}}
        = 
        f(z).
    \end{equation*}
    In the following, we assume $p\colon \RR \to \RR$ with $p(0) = 0$ is strictly monotone, continuously differentiable and surjective, such that \Cref{as:constructions} is satisfied.
    The derivatives of $f$ and $\eta$ read as 
    \begin{align*}
        Df(z) 
        = 
        \vec{ 
            0 & - \partial_x \\ 
            - \partial_x \circ p'(z_1) & - \gamma F'(z_2) \visc{+ \nu \partial_x^2}
            },~~
        D\eta(z) 
        = 
        \vec{ 
            p'(z_1) & 0 \\ 
            0 & 1
            }
    \end{align*}
    and we obtain $Df(z) D\eta(z)\inv =  \tvec{ 0 & - \partial_x \\ - \partial_x & - \gamma F'(z_2) \visc{+ \nu \partial_x^2} }$.
    Similar to \Cref{rem:P-A-definition-of-J-R}, we now have $f(z) = M(z) \eta(z)$ with
    \begin{equation*}
        M(z) 
        = 
        \int_{0}^{1} Df(sz) D\eta(sz)\inv \ds 
        = 
        \vec{ 
            0 & - \partial_x \\ 
            - \partial_x & - \gamma \frac{F(z_2)}{z_2} \visc{+ \nu \partial_x^2}
            },
    \end{equation*}
    where we have used $\int_{0}^{1} F'(s z_2) \ds = \frac{F(z_2)}{z_2}$. 
    Note that we have the formal adjoint\visc{s} $(\partial_x)^*= -\partial_x$\visc{ and $(\partial_x^2)^* = \partial_x^2$}, so that $\partial_x + (\partial_x)^* = 0$\visc{ and $\partial_x^2 + (\partial_{x}^2)^* = 2 \partial_x^2$}.
    With $M(z)\skew = \tfrac12 (M(z) - M(z)^*)$ and $M(z)\sym = \tfrac12 (M(z) + M(z)^*)$ we now obtain 
    \begin{align*}
        J = M(z)\skew & = 
        \vec{ 
            0 & - \partial_x \\ 
            - \partial_x & 0
            },
        \\
        R(z) = -M(z)\sym & = 
        \vec{ 
            0 & 0 \\ 
            0 & \gamma \frac{F(z_2)}{z_2} \visc{- \nu \partial_x^2}
        }.
    \end{align*}
    We observe that $R(z)$ is formally semidefinite in the sense that 
    \begin{equation*}
        \int_{\Omega} \Big(\vec{v \\ w}, R(z) \vec{v \\ w}\Big) \dx 
        = 
        \gamma \int_{\Omega} w^2 \tfrac{F(z_2)}{z_2} \dx
        \visc{+ \nu \int_{\Omega} (\partial_{x} w)^2 \dx}
        \geq 0
    \end{equation*}
    for smooth, compactly supported functions $v$ and $w$, where $\frac{F(z_2)}{z_2} \geq 0$ because $F$ is odd.
    Further, we remark that $J$ and $R(z)$ as above coincide with the usual decomposition $f(z) = (J - R(z))\eta(z)$ for this example.
\end{example}

\section{Conclusion}\label{sec:conclusion}

In this paper, we have investigated the relationship between passivity encoding representations and passivity for nonlinear systems, offering methods to construct \portHamiltonian representations when both the system dynamics and the associated Hamiltonian are known.
Under the assumption of injectivity of~$\eta$, we demonstrated that every passive system can be expressed in the form~\eqref{eq:ph-b}.
For \pH systems of the form~\eqref{eq:ph-a}, we highlighted that the semidefiniteness of the dissipation matrix on the entire state space is an intrinsic property of the model class that does not necessarily follow from passivity. 
As a remedy, we provided conditions that ensure a representation in the form~\eqref{eq:ph-a} is feasible, leading to a systematic method of constructing port-Hamiltonian representations. 
We successfully applied our approach to multiple examples and observed that it is also feasible for some infinite dimensional systems.

As storage functions are not unique~\cite{willems72-dissipative1}, an open question is choosing a storage function $\ham$ such that an associated \portHamiltonian representation exists and is easy to obtain.
Other open topics are incorporating direct feedthrough terms and the rigorous treatment of the infinite dimensional case.

\section*{Acknowledgment}
We thank the reviewers for several helpful comments.
Further, T.\,Breiten and A.\,Karsai thank the Deutsche Forschungsgemeinschaft for their support within the subproject B03 in the Sonderforschungsbereich/Transregio 154 “Mathematical Modelling, Simulation and Optimization using the Example of Gas Networks” (Project 239904186).

\bibliographystyle{siam}
\bibliography{nlph}

\end{document}